\documentclass[a4paper,12pt]{article}
\usepackage{amsmath,amsfonts,enumerate,amssymb,amsthm,color}
\usepackage{tikz, tikzscale}
\usepackage{ifthen}
\usetikzlibrary{arrows, calc,matrix, decorations.pathmorphing}
\title{Vertex-transitive Haar graphs that are not Cayley graphs}
\author{Marston Conder,$^{a}$ \;Istv\'an Est\'elyi,$^{b,c}$   \; Toma\v{z} Pisanski$^{d,c,b}$ \\  [+0.75ex]
$^a$ {\small Mathematics Department, University of Auckland, Private Bag 92019, Auckland 1142, New Zealand} \\ [-0.5ex]
$^b$ {\small FMF, University of Ljubljana, Jadranska 19, 1000 Ljubljana, Slovenia} \\ [-0.5ex]
$^c$ {\small IAM, University of Primorska, Muzejski trg 2, 6000 Koper, Slovenia}\\[-0.5ex] 
$^d$ {\small FAMNIT, University of Primorska, Glagolja\v{s}ka 8, 6000 Koper, Slovenia} \\ [-0.5ex]
}
\date{}

\newtheorem{thm}{Theorem}

\newtheorem{cor}[thm]{Corollary}
\newtheorem{prop}[thm]{Proposition}
\newtheorem{prob}{Problem}

\def\Z{\mathbb{Z}}
\DeclareMathOperator{\cay}{Cay}

\DeclareMathOperator{\aut}{Aut}

\DeclareMathOperator{\bcay}{BCay}

\newcommand{\comment}[1]{}

\makeatletter
\providecommand*{\xmapstofill@}{%
  \arrowfill@{\mapstochar\relbar}\relbar\rightarrow
}
\providecommand*{\xmapsto}[2][]{%
  \ext@arrow 0395\xmapstofill@{#1}{#2}%
}
\makeatother

\begin{document}

\maketitle

\let\thefootnote\relax\footnote{
The first author was supported by New Zealand's Marsden Fund (grant no.\ UOA1323). 
The second author was
supported in part by the Slovenian Research Agency (research program P1-0294, research project J1-7051 and a Young Researchers Grant).
The third author was
supported in part by the Slovenian Research Agency (research program P1-0294 and research projects N1-0032, L7-5554 and J1-6720), and in part by the H2020 Teaming InnoRenew CoE.
\\  [+0.5ex]
{\em  E-mail addresses:} m.conder@auckland.ac.nz (Marston Conder), istvan.estelyi@student.fmf.uni-lj.si (Istv\'an Est\'elyi),  tomaz.pisanski@upr.si (Toma\v{z} Pisanski).
}

${}$\\[-30pt] 
\centerline{\em Dedicated to Egon Schulte and K\'{a}roly Bezdek on the occasion of their 60th birthdays}
\\[-6pt] 

\begin{abstract}
In a recent paper 
(arXiv:1505.01475 ) Est\'elyi and Pisanski raised a question whether there exist vertex-transitive Haar graphs that are not Cayley graphs. In this note we construct an infinite family of trivalent Haar graphs that are vertex-transitive but non-Cayley. The smallest example has 40 vertices and is the well-known Kronecker cover over the dodecahedron graph $G(10,2)$, 
occurring as the graph `{\bf 40}' in the Foster census  of connected symmetric trivalent graphs. 
\medskip

 \medskip

\end{abstract} 

\section{Introduction}
\label{sec:Intro}

Let $G$ be a group, and $S$ be a subset of $G$ with $1_G \notin S$. 
Then the {\em Cayley graph\/} $\cay(G,S)$ is 
is the graph with vertex-set $G$ and with edges of the form $\{g,sg\}$ for all $g \in G$ and $s \in S$. 
Equivalently, since all edges can be written in the form $\{1,s\}g$, 
this is a covering graph over a single-vertex graph having loops and semi-edges, 
with voltages taken from $S${\hskip 1pt}: the order of a voltage over a semi-edge is $2$ 
(corresponding to an involution in $S$), while the order of voltage over a loop is greater than $2$ 
(corresponding to a non-involution in $S$).  Note that we may assume $S = S^{-1}$. 

A natural generalisation of Cayley graphs are the so called \emph{Haar graphs}, introduced in \cite{HlaMP02} by Hladnik  et al, as follows. 
A {\em dipole\/} is a graph with two vertices, say black and white, and parallel edges (each from the \emph{white} vertex to the \emph{black} vertex), but no loops.
Given a group $G$ and an arbitrary subset $S$ of $G,$ the \emph{Haar graph} $H(G,S)$ is the regular $G$-cover of a dipole with $|S|$ parallel edges, labeled by elements of $S$. In other words, the vertex-set of $H(G,S)$ is  
$G \times \{0,1\},$ and the edges are of the form $\{(g,0),(sg,1)\}$ for all $g \in G$ and $s \in S$. 
If it is not ambiguous, we use the notation $(x,0) \sim (y,1)$ to indicate an edge $\{(x,0),(y,1)\}$ of $H(G,S)$.  
The name `Haar graph' comes from the fact that when $G$ is an abelian group, the Schur norm of 
the corresponding adjacency matrix can be easily evaluated via the so-called Haar integral on G (see~\cite{Hla99}).

Note that the group $G$ acts on $H(G,S)$ as a group of automorphisms, by right multiplication, 
and moreover, $G$ acts regularly on each of the two parts of $H(G,S)$, namely $\{(g,0) : g \in G\}$ and $\{(g,1) : g \in G\}$. 
Conversely, if $\Gamma$ is any bipartite graph and its automorphism group $\aut\Gamma$ has a subgroup $G$ 
that acts regularly on each part of $\Gamma$, then $\Gamma$ is a Haar graph --- indeed $\Gamma$ 
is isomorphic to $H(G,S)$ where $S$ is determined by the edges incident with a given vertex of $\Gamma$. 

Haar graphs form a special subclass of the more general class of \emph{bi-Cayley graphs}, which are graphs that admit a semiregular group of automorphisms with two orbits of equal size. Every bi-Cayley graph can be realised as follows. Let $L$ and $R$ be subsets of a group $G$ such that $L=L^{-1}$, $R=R^{-1}$ and $1\notin L\cup R,$ and let $S$ be any subset of $G$. Now take a dipole with edges labelled by elements of $S$, and add $|L|$ loops to the white (or `left') vertex and label these by elements of $L$, and similarly add $|R|$ loops to the black (or `right') vertex and label these by elements of $R$. This is a voltage graph, and the bi-Cayley graph $\bcay(G,L,R,S)$ is its regular $G$-cover. The vertex-set of $\bcay(G,L,R,S)$ is $G \times \{0,1\},$ and the edges are of three forms: $\{(g,0),(lg,0)\}$ for $\,l\in L$, $\,\{(g,1),(rg,1)\}$ for $\,r\in R,$ and $\{(g,0),(sg,1)\}$ for $s \in S,$  for all $g \in G$.  Note that the Haar graph $H(G,S)$ is exactly the same as the bi-Cayley graph $\bcay(G,\emptyset,\emptyset,S)$.

Recently bi-Cayley graphs (and Haar graphs in particular) have been investigated by several authors --- see \cite{DuX00,EstP15, ExoJ11, HlaMP02, JinL10,KoiK14,KKP14,KKM10, KovMMM09,Lu03,LuWX06,Pis07,W08,ZhouF14}, for example. 

It is known that every Haar graph over an abelian group is a Cayley graph (see~\cite{Lu03}). 
More precisely, if $A$ is an abelian group, 
then a Haar graph over $A$ is a Cayley graph over the corresponding \emph{generalised dihedral group} $D(A),$ 
which is the group generated by $A$ and the automorphism of $A$ that inverts every element of $A$ (see \cite{Sco64}).  
The authors of \cite{HlaMP02} considered only cyclic Haar graphs --- that is, Haar graphs $H(G,S)$ where $G$ is a cyclic group.
In \cite{EstP15}, the second and third authors of this paper extended the study of Haar graphs to those over non-abelian groups, and found some that are not vertex-transitive, and some others that are Cayley graphs. The existence of Haar graphs that are vertex-transitive but non-Cayley remained open, and led to the following question. 

\begin{prob}\label{prob2}
Is there a non-abelian group $G$ and a subset $S$ of $G$ such that the Haar graph $H(G,S)$ is vertex-transitive but
non-Cayley?
\end{prob}

In this note we give a positive answer to the above question, by exhibiting an infinite family 
of trivalent examples, coming from a family of double covers of generalised Petersen graphs. 
These graphs, which we denote by $D(n,r)$ for any integers $n$ and $r$ with $n \ge 3$ and $0 \! < \! r \! < \! n$, 
are described in Section~\ref{sec:DefineFamily}.  They have been considered previously 
by other authors (as we explain); in particular, by a theorem of Feng and Zhou \cite{ZhouF14}, 
it is known exactly which of the graphs $D(n,r)$ are vertex-transitive, and which are Cayley.  
Then in Section~\ref{sec:Main} we determine necessary and sufficient conditions for $D(n,r)$ to be a Haar graph, 
and this provides the answer in Section~\ref{sec:Final}.  

\section{The graphs $D(n,r)$ and their properties}
\label{sec:DefineFamily}

Let $G(n,r)$ be the generalised Petersen graph on $2n$ vertices with span $r$. 
By $D(n,r)$ we denote a double cover of $G(n,r),$ in which the edges get non-trivial voltage 
if and only if they belong to the `inner rim' (see below). 
This gives a class of graphs that was introduced by Zhou and Feng \cite{ZhouF12} under the 
name of \emph{double generalised Petersen graphs}, and studied recently also by 
Kutnar and Petecki \cite{KutP1x}.  
In both \cite{ZhouF12} and \cite{KutP1x}, the notation $DP(n,r)$ was used for the graph $D(n,r)$. 
But it is easy for us to define the vertices and edges of the graph $D(n,r)$ explicitly. 

\smallskip
There are four types of vertices, called $u_i,v_i,w_i$ and $z_i$ (for $i \in \Z_n$), 
and three types of edges, given by the sets \\[-22pt] 
\begin{align*}
\Omega &=\{\{u_i,u_{i+1}\}, \{z_i,z_{i+1}\} : i\in \mathbb{Z}_n\} &\textrm{(the `outer' edges),}\\
\Sigma &=\{\{u_i,v_i\}, \{ w_i,z_i\} : i\in \mathbb{Z}_n\} &\textrm{(the `spokes'), and }\\
I &=\{\{v_i,w_{i+r}\}, \{ v_i,w_{i-r}\} : i\in \mathbb{Z}_n\} &\textrm{(the `inner' edges).}
\end{align*} 

This specification makes it easy to see that each $D(n,r)$ a special  tetracirculant~\cite{FreK13}, 
which is a cyclic cover $\Sigma_0(n,a,k,b)$ over the voltage graph given in Figure 1. 
To see this, simply take $a = b = 1$ and $k = 2r$, and then $D(n,r) \cong \Sigma_0(n,1,2r,1)$. 
\\[-18pt] 

\begin{figure}[!htb]
\label{fig:voltage_graph}
\begin{center}
\includegraphics[]{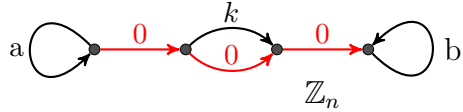}
${}$ \vskip -36pt ${}$
\end{center}
\centering
\caption{Voltage graph defining the tetracirculant $\Sigma_0(n,a,k,b)$}
\end{figure}

We now give some of other properties of the graphs $D(n,r)$.

\begin{prop}
Every $D(n,r)$ is connected.
\end{prop}

\begin{proof}
Clearly all of the $u_i$ lie in the same component as each other, as do all the $z_j$. 
Next, all the $v_i$ lie in the same component as the $u_i$, and similarly, all the $w_j$ lie 
in the same component as the $z_j$. 
Finally, there are edges between the vertices $v_i$ and some of the $w_j$, 
and this makes the whole graph connected.
\end{proof}

\begin{prop}
\label{dpbip}
The graph $D(n,r)$ is bipartite if and only if $n$ is even.
\end{prop}

\begin{proof}
If $n$ is odd, then the vertices $u_i$ lie in a cycle of odd length, and so the graph is not bipartite. 
On the other hand, if $n$ is even, then the graph is bipartite, with one part containing the vertices $u_i$ 
and $w_{i \pm r}$ for even $i$ and the vertices $v_j$ and $z_{j \pm r}$ for odd $j$. 
\end{proof}

We now consider automorphisms of the graphs $D(n,r)$. 
Some automorphisms are apparent from the definition, such as these, which were noted in \cite{KutP1x}: 
\\[-6pt]

\begin{tabular}{rllllr}
$\alpha :$ & \quad $u_i\mapsto u_{i+1},$ & \ $v_i\mapsto v_{i+1},$ & \ $w_i\mapsto w_{i+1},$ & \ $z_i\mapsto z_{i+1}$ & (rotation), \\[+4pt]
$\beta :$ & \quad $u_i\mapsto z_{i},$ & \ $v_i\mapsto w_{i},$ & \ $w_i\mapsto v_{i},$ & \ $z_i\mapsto u_{i}$ & \quad (flip symmetry), \\[+4pt] 
$\gamma :$ & \quad $u_i\mapsto u_{-i},$ & \ $v_i\mapsto v_{-i},$ & \ $w_i\mapsto w_{-i},$ & \ $z_i\mapsto z_{-i}$ & (reflection). \\[+12pt] 
\end{tabular}

\noindent 
Immediately we obtain the following:

\begin{prop}
\label{D(n,r)-2orbits} 
The automorphism group of the graph $D(n,r)$ has at most two orbits on vertices, 
namely the set of all $u_i$ and all $z_j$, and the set of all $v_i$ and all $w_j$. 
\end{prop}

Note also that $\alpha$ and $\beta$ commute with each other. 
In fact, Zhou and Feng \cite{ZhouF12} proved that $D(n,r)$ is isomorphic 
to the bi-Cayley graph $\bcay(G, R,L,\{1\})$ over the abelian 
group $G=\langle\alpha,\beta\rangle\cong \Z_n\times \Z_2,$ 
and  $R=\{\alpha, \alpha^{-1}\}$ and $L=\{\alpha^r\beta,\alpha^{-r}\beta\}$.

\smallskip
Next, we consider isomorphisms among the graphs $G(n,r)$ and $D(n,r)$. 

\begin{prop}\label{rm-r}
For every $n$ and $r$, the graph $D(n,r)$ is isomorphic to $D(n,n\!-\!r)$, 
and $D(2n,r)$ is isomorphic to  $D(2n,n\!-\!r)$.
\end{prop}

\begin{proof}
First, the graphs $D(n,r) \cong D(n,n\!-\!r)$  because $G(n,r)$ is identical to $G(n,n\!-\!r)$. 
For the second part, consider a 180 degree rotation of the two `inner' layers, 
namely $w_i \mapsto w_{i+n}$ and $z_i \mapsto z_{i+n}$ for all $i$.  
This shows that $D(2n,r)$ is isomorphic to $D(2n,n+r)$, and then applying the first part 
gives $D(2n,r) \cong D(2n,2n-(n\!+\!r))= D(2n,n\!-\!r)$. 
\end{proof}

Here we note that it can happen that the graphs $D(n,r)$ and $D(n,s)$ are different 
when $G(n,r)$ is isomorphic to $G(n,s)$. 
For instance, $G(7,2)$ is isomorphic to $G(7,3)$ but $D(7,2)$ is not isomorphic to $D(7,3)$,  
since $D(7,3)$ is planar but $D(7,2)$ is not. 

Also we have the following:

\begin{prop}
For every $r$, the graph $D(2r\!+\!1,r)$ is planar, 
and isomorphic to the generalised Petersen graph $G(4r\!+\!2,2)$.
\end{prop}

\begin{proof} 
To see that $D(2r\!+\!1,r)$ is planar, first note that since $r$ is coprime to $2r\!+\!1$, 
the edges between the vertices $v_i$ and $w_j$ give a cycle of length $2(2r\!+\!1)$, 
namely $(v_0,w_{-r},v_{1},w_{1-r},v_{2},w_{2-r},\dots,v_{-2},w_{r-1},v_{-1},w_{r})$.  
Now draw three concentric circles, with the middle one for this $2(2r\!+\!1)$-cycle, 
the inside one for the $(2r+1)$-cycle $(u_0,u_1,\dots,u_{2r})$, and the outside one for 
the $(2r+1)$-cycle $(z_0,z_1,\dots,z_{2r})$, in a consistent order, and then add the 
spoke edges $\{u_i,v_i\}$ and $\{w_i,z_i\}$ in the natural way.  
In the resulting planar drawing of $D(2r\!+\!1,r)$, there is 
an inner face of length $2r\!+\!1$ (with the $u_i$ as vertices), then two layers of pentagonal faces 
(bounded by cycles of the form $(u_i,v_i,w_{i-r},v_{i+1},u_{i+1})$ and  $(v_j,w_{j+r},z_{j+r},z_{j-r},w_{j-r})$), 
and an outer face of length $2r\!+\!1$ (with the $z_j$ as vertices). 
After doing this, it is also easy to see that $D(2r\!+\!1,r)$ is isomorphic to the generalised 
Petersen graph $G(4r\!+\!2,2)$, with the spoke edges joining vertices of the 
large $2(2r+1)$-cycle (on the vertices $v_i$ and $w_j$) to the two $(2r+1)$-cycles 
(on the vertices $u_i$ and vertices $z_j$ respectively). 
\end{proof}

In particular, the graph $D(5,2)$ is isomorphic to the dodecahedral graph $G(10,2)$, 
and hence $D(5,2)$ is vertex-transitive.  But as we will see, it is not a Haar graph. 

\smallskip
Finally in this section, we consider the questions of which of the graphs $D(n,r)$ are vertex-transitive, 
and which are Cayley (or equivalently, for which $\aut(D(n,r))$ has a subgroup that acts 
regularly on vertices).  
Recall that $\aut(D(n,r))$ has at most two orbits on vertices, and just one when $(n,r) = (5,2)$. 
The complete picture was determined by Feng and Zhou in~\cite[Theorem 1.3]{ZhouF14}, as follows:

\begin{thm}
\label{zhoufvt}
The graph $D(n,r)$ is vertex-transitive if and only if 
$n = 5$ and $r = \pm 2$, or $n$ is even and $r^2\equiv \pm 1 $ mod $\frac{n}{2}$.
In the first case, $D(n,r)$ is isomorphic to the dodecahedral graph $G(10, 2)$, 
which is non-Cayley, and in the second case, if $\,r^2\equiv 1$ mod~$\frac{n}{2}$ then $D(n,r)$ is a Cayley graph, 
while if $\,r^2\equiv -1$ mod~$\frac{n}{2}$ then $D(n,r)$ is non-Cayley.
\end{thm}

\section{The graphs $D(n,r)$ as Haar graphs}
\label{sec:Main}

Recall that a Haar graph is a regular cover of a dipole, and also a bi-Cayley graph. 
Also we have the following, proved in a different way in \cite[Proposition 5]{EstP15}: 

\begin{prop}\label{cayhaar}
A Cayley graph is a Haar graph if and only if it is bipartite.
\end{prop}

\begin{proof} 
Let $\Gamma$ be a Cayley graph, say for a group $K$.  
Then $K$ acts on $\Gamma$ as a group of automorphisms, and acts regularly 
on the vertices of $\Gamma$. 
Now if $\Gamma$ is a Haar graph, then by definition $\Gamma$ is bipartite.
Conversely, suppose $\Gamma$ is bipartite.  Then the subgroup $G$ of $K$ 
preserving each of the two parts of $\Gamma$ has index $2$ in $K$, and acts regularly 
on each part, so $\Gamma$ is a Haar graph (by the argument given in the Introduction). 
\end{proof} 

We can now prove our main theorem: 

\begin{thm}
\label{thm:Main}
$D(n,r)$ is a Haar graph if and only if it is vertex-transitive and $n$ is even. 
\end{thm}

\begin{proof}
First, we note that $D(n,r)$ is bipartite if and only if $n$ is even, by Proposition \ref{dpbip}, 
and hence we may suppose that $n$ is even, and then show that under that assumption, 
$D(n,r)$ is a Haar graph if and only if it is vertex-transitive. 

One direction is easy.  Suppose $\Gamma = D(n,r)$ is a Haar graph, say $H(G,S)$. 
Then by the definition of a Haar graph given in the Introduction, the subgroup $G_R$ 
of $\aut \Gamma$ induced by $G$ has two orbits on vertices, namely the two parts 
of the bipartition of~$\Gamma$.  On the other hand, by Proposition~\ref{D(n,r)-2orbits}, 
all the vertices $u_i$ lie in the same orbit of $\aut \Gamma$; and then since these vertices 
lie in both parts of $\Gamma$, it follows that $\aut \Gamma$ has a single orbit on vertices. 
Thus $\Gamma$ is vertex-transitive. 

For the converse, suppose that $\Gamma = D(n,r)$ is vertex-transitive, and let $m = \frac{n}{2}$. 
Then by Theorem \ref{zhoufvt}, we know that $r^2 \equiv \pm 1$ mod $m$. 
Also by Proposition \ref{rm-r} we may suppose that $0 < r < m$, 
and further, we may suppose that $r$ is odd, because if $r$ is even then 
$m$ is odd, and then by Proposition \ref{rm-r} we can replace $r$ by $m-r$. 
We now proceed by considering separately the two cases $r^2 \equiv \pm 1$ mod $m$.  

\smallskip
Case (a): Suppose that $r^2 \equiv 1$ mod $m$.  Then by Theorem \ref{zhoufvt}, 
we know that $D(n,r)$ is a Cayley graph, and also since it is bipartite, it follows from 
Proposition \ref{cayhaar} that it is a Haar graph as well. 

\smallskip
Case (b): Suppose that $r^2 \equiv -1$ mod $m$.  In this case we construct a group 
of automorphisms of $D(n,r)$ that acts regularly on each part of $D(n,r)$.  To do this, 
we take the automorphism $\alpha$ from the previous section, given by \\[-9pt] 
$$
\alpha : \ u_i\mapsto u_{i+1}, \ \ v_i\mapsto v_{i+1}, \ \ w_i\mapsto w_{i+1}, \ \ z_i\mapsto z_{i+1},
$$
and then take an additional automorphism $\delta$, given by \\[-6pt] 

\noindent 
\begin{tabular}{ll} 
$\delta : \ u_i\mapsto v_{ri+1}, \ \ v_i\mapsto u_{ri+1}, \ \ w_i\mapsto z_{ri+1}, \ \ z_i\mapsto w_{ri+1}$  
 &  if $m$ is odd and $i$ is even, \\[+2pt] 
$\delta : \ u_i\mapsto w_{ri+1}, \ \ v_i\mapsto z_{ri+1}, \ \ w_i\mapsto u_{ri+1}, \ \ z_i\mapsto v_{ri+1}$  
 &  if $m$ is odd and $i$ is odd, \\[+7pt] 
\end{tabular} 

\noindent or  \\[-6pt] 

\noindent 
\begin{tabular}{ll} 
$\delta : \ u_i\mapsto v_{ri+1}, \  v_i\mapsto u_{ri+1}, \  w_i\mapsto z_{ri+m+1}, \  z_i\mapsto w_{ri+m+1}$ 
 &  if $m$ is even and $i$ is even, \\[+2pt] 
$\delta : \ u_i\mapsto w_{ri+1}, \  v_i\mapsto z_{ri+1}, \  w_i\mapsto u_{ri+m+1}, \  z_i\mapsto v_{ri+m+1}$ 
 &  if $m$ is even and $i$ is odd.  \\[+8pt] 
\end{tabular} 

It is a straightforward exercise to verify that $\delta$ preserves the edge-set $\Omega \cup \Sigma \cup I$ 
of $D(n,r)$, and also preserves the two parts of $D(n,r)$, given in the proof of Proposition~\ref{dpbip}. 
To do the former, it is important to note that $r^2 \equiv 1$ mod $4$ (because $r$ is odd), 
and hence that $r^2 \equiv -1$ mod~$n$ when $m$ is odd, while $r^2 \equiv m\!-\!1$ mod $n$ when $m$ is even. 
For example,  if $m$ and $i$ are even then $\{v_i,w_{i+r}\}^\delta = \{u_{ri+1},u_{r(i+r)+m+1}\} = \{u_{ri+1},u_{ri}\}$. 

\smallskip
It is also easy to see that conjugation by $\delta$ takes $\alpha^2$ to $\alpha^{2r}$, 
and so the subgroup $G$ of $\aut(D(n,r))$ generated by $\alpha^2$ and $\delta$ 
is isomorphic to the semi-direct product $\Z_{m} \rtimes_r \Z_4$.  
In particular, $G$ has order $4m = 2n$.  
Also $G$ acts transitively and hence regularly on each of the two parts of $D(n,r)$, 
and therefore $D(n,r)$ is a Haar graph. 
\end{proof}

\section{Vertex-transitive Haar graphs that are not Cayley graphs}
\label{sec:Final}

Combining Theorems \ref{zhoufvt} and \ref{thm:Main}, we have the following, 
in answer to Problem \ref{prob2}: 

\begin{thm}
\label{thm:Final}
${}$\\[-22pt]
\begin{enumerate}
\item[{\rm (a)}] If $n$ is odd, or if $n$ is even and $r^2\not\equiv \pm 1 $ mod $\frac{n}{2}$, 
   then $D(n,r)$ is not a Haar graph, and is vertex-transitive only when $(n,r) = (5,\pm 2);$  \\[-22pt]
\item[{\rm (b)}] If $n$ is even and $r^2\equiv 1$ mod $\frac{n}{2}$, 
   then $D(n,r)$ is a Haar graph and a Cayley graph$\hskip 1pt ;$  \\[-22pt]
\item[{\rm (c)}] If $n$ is even and $r^2\equiv -1$ mod $\frac{n}{2}$, 
   then $D(n,r)$ is a Haar graph and is vertex-transitive but not a Cayley graph.
\end{enumerate}
\end{thm}



\begin{cor}
If $m > 2$ and $r^2\equiv -1$ mod $m$, then $D(2m,r)$ 
is a Haar graph that is vertex-transitive but non-Cayley.  
In particular, there are infinitely many such graphs. 
\end{cor} 

\begin{proof}
The first part follows immediately from Theorem \ref{thm:Final}, 
and the second part follows from a well known fact in number theory, 
namely that $-1$ is a square mod $m$ if and only if $m$ or $m/2$ is 
a product of primes $p\equiv 1$ mod $4$ (see \cite[Chapter 6]{HW08}), 
or simply by taking $m = r^{2}+1$ for each integer $r \ge 2$. 
\end{proof} 

%
%
%

We discovered the first few of these examples during the week of the 
conference {\em Geometry and Symmetry}, held in 2015 at Veszpr{\' e}m, Hungary,  
to celebrate the 60th birthdays of K\'{a}roly Bezdek and Egon Schulte. 

The smallest of our examples is $D(10,2)$, of order $40$, 
occurring when $m = 5$ and $r \equiv \pm 2$ or $\pm 3$ mod $10$ 
(noting that $m-r = 3$ when $(m,r) = (5,2)$). 
This is also the smallest known Haar graph that is vertex-transitive and non-Cayley. 
It is a Kronecker cover over the dodecahedral graph $G(10,2)$, and is also a double 
cover over the Desargues graph $G(10,3)$.  
These graphs are illustrated in Figure 2. 

%

\medskip

\begin{figure}[!htb]
\label{fig:smallgraphs}
\begin{center}
\includegraphics[]{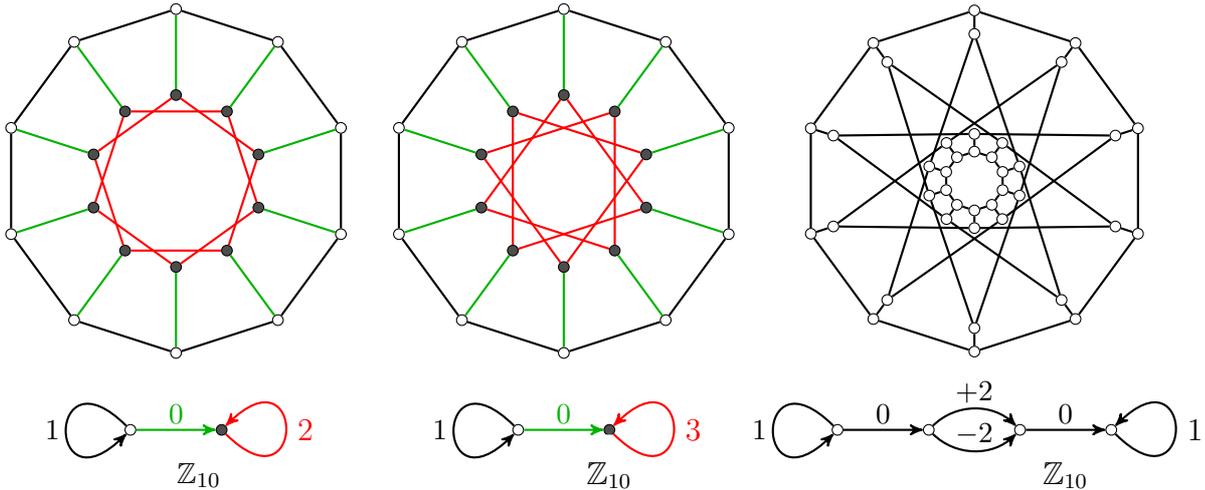}
\end{center}
${}$\\[-54pt] 
${}$
\centering
\caption{The dodecahedral graph $G(10,2)$, the Desargues graph $G(10,3)$, 
and the Haar graph $D(10,2) \cong D(10,3) \cong$ {\bf F40}}
\end{figure}

\smallskip 

The graph $D(10,2)$ was known by R.M. Foster as early as the late 1930s, 
and appears as the graph `{\bf 40}' (alternatively known as `{\bf F40}') in the {\em Foster Census\/} of connected symmetric 
trivalent graphs \cite{Fos88}.   
It was also studied in \cite{Wei84} by Asia Ivi\'c Weiss (the chair of the Veszpr{\' e}m conference), 
and by Betten, Brinkmann and Pisanski in~\cite{BetBP00}, and Boben, Gr\" unbaum, Pisanski and \v{Z}itnik 
in \cite{BobGPZ06}. 
It has girth $8$, and automorphism group of order $480$, and it is not just vertex-transitive, but is also arc-transitive.  
Moreover, by a very recent theorem of Kutnar and Petecki \cite{KutP1x}, the graph $D(n,r)$ 
is arc-transitive only when $(n,r) = (5,2)$ or $(10,2)$ or $(10,3)$.  
This implies that {\bf F40} is the only example from the family of graphs $D(n,r)$ that is arc-transitive 
but non-Cayley.  

In fact, {\bf F40} is the smallest vertex-transitive non-Cayley Haar graph, in terms of both the graph order 
and the number of edges. 
We found this by running a {\sc Magma} \cite{BosCP97} computation to construct all Haar graphs 
with at most $40$ vertices or at most $60$ edges, with a check for which of the 
graphs are vertex-transitive but non-Cayley.  
Incidentally, this computation shows that there are $60$ different examples of order $40$, 
with valencies running between $3$ and $17$, but just one of valency $3$, namely {\bf F40}. 

Finally, there are many other examples of vertex-transitive non-Cayley Haar graphs that 
are not of the form $D(n,r)$, including $3$-valent examples of orders $80$, $112$, $120$ and $128$, 
and higher-valent examples of orders $48$, $64$, $72$, $78$ and $80$. 
Among the $3$-valent examples, many are arc-transitive, 
including the graphs {\bf F80} and {\bf F640} in the Foster census \cite{Fos88} and its extended 
version in \cite{CoDo02,CoNe09}, and others in the first author's complete set of all 
arc-transitive trivalent graphs of order up to 10000 described on his website \cite{Co-list}. 
Most of these `small' examples are abelian regular covers of {\bf F40}, 
of orders 1280, 2560, 3240, 5000, 5120, 6480, 6720, 9720 and 10000, 
and are $3$-arc-regular, but two others are $2$-arc-regular of type $2^2,$ with orders 6174 and 8064, 
and these are abelian regular covers of the Pappus graph ({\bf F18}) and the Coxeter graph ({\bf F28}) respectively.




\section*{Acknowledgements} 

The authors acknowledge with gratitude the use of 
the {\sc Magma} system \cite{BosCP97} for helping to find and analyse examples, 
and would also like to thank Klavdija Kutnar and Istv\'{a}n Kov\'{a}cs for fruitful conversations and for pointing out several crucial references. 



\begin{thebibliography}{100}

{\small 
%

\bibitem{BetBP00}  A. Betten, G.  Brinkmann and T.  Pisanski. Counting symmetric configurations v3.
Proceedings of the 5th Twente Workshop on Graphs and Combinatorial Optimization (Enschede, 1997), 
{\em Discrete Appl. Math.}  99  (2000),  no. 1-3, 331--338.
\\[-20pt] 

\bibitem{BobGPZ06}  M. Boben, B. Gr\" unbaum, T.  Pisanski and A. \v{Z}itnik, 
Small triangle-free configurations of points and lines, 
{\em Discrete Comput. Geom.}  35  (2006),  no. 3, 405--427.
\\[-20pt] 

\bibitem{BosCP97}  W. Bosma, J. Cannon and C. Playoust,
The {\sc Magma} Algebra System I: The User Language, 
{\em J. Symbolic Comput.} 24 (1997), 235--265.
\\[-20pt] 

\bibitem{Co-list}  M.D.E. Conder, 
{\em Trivalent $($\hskip -1pt cubic\hskip +1pt$)$ symmetric graphs on up to 10000 vertices}, 
available at 
{\tt http://www.math.auckland.ac.nz/$\sim$conder/symmcubic10000list.txt.} 
\\[-20pt] 

\bibitem{CoDo02}  M. Conder and P. Dobcs\'anyi, 
Trivalent symmetric graphs on up to 768 vertices, 
{\em J. Combin. Math. Combin. Comput.} 40 (2002), 41--63. 
\\[-20pt] 

\bibitem{CoNe09}  M. Conder and R. Nedela, 
A refined classification of symmetric cubic graphs, 
{\em J. Algebra\/} 322 (2009), 722--740.
\\[-20pt] 

\bibitem{Fos88} I.Z. Bouwer, W.W. Chernoff, B. Monson, B and Z. Star, 
{\em The Foster Census}, Charles Babbage Research Centre, 1988. 
\\[-20pt] 

\bibitem{DuX00}  S.F. Du and M.Y. Xu,
A classification of semi-symmetric graphs of order $2pq,$ 
{\em Comm. Algebra\/} 28 (6) (2000), 2685--2715. 
\\[-20pt] 

\bibitem{EstP15}  I. Est\'elyi and T. Pisanski, 
Which Haar graphs are Cayley graphs?, arXiv:1505.01475.
\\[-20pt] 

\bibitem{ExoJ11}  G. Exoo and R. Jajcay, 
On the girth of voltage graph lifts,
{\em Europ. J. Combin.} 32 (2011) 554--562. 
\\[-20pt] 

\bibitem{FreK13}  B. Frelih and K. Kutnar, 
Classification of symmetric cubic tetracirculants and pentacirculants,
{\em Europ. J. Combin.} 34 (2013) 169--194. 
\\[-20pt] 



\bibitem{HW08} G.H. Hardy and E.M. Wright, 
{\em An Introduction to the Theory of Numbers}, 
6th ed., Oxford Univ. Press, 2008. 
\\[-20pt] 

\bibitem{Hla99}  M. Hladnik, 
Schur norms of bicirculant matrices, 
{\em Lin. Alg. Appl.} 286 (1999), 261--272.
\\[-20pt] 

\bibitem{HlaMP02}  M. Hladnik, D. Maru\v{s}i\v{c} and T. Pisanski, 
Cyclic Haar graphs, 
{\em Discrete Math.} 244 (2002), 137--153.
\\[-20pt] 

\bibitem{JinL10}  W. Jin and W. Liu,
A classification of nonabelian simple $3$-BCI-groups,
{\em Europ. J. Combin.} 31 (2010), 1257--1264.
\\[-20pt] 

\bibitem{KoiK14}  H. Koike and I. Kov\'acs, 
Isomorphic tetravalent cyclic Haar graphs, 
{\em Ars Math. Contemp.} 7 (2014), 215--235.
\\[-20pt] 

\bibitem{KKP14}  H. Koike, I. Kov\'{a}cs and T. Pisanski, 
The number of cyclic configurations of type $(v_3)$ and the isomorphism problem, 
{\em J. Combin. Des.}  22  (2014),  no. 5, 216--229.
\\[-20pt] 

\bibitem{KKM10}  I. Kov\'{a}cs, K. Kutnar and D Maru\v{s}i\v{c}, 
Classification of edge-transitive rose window graphs, 
{\em J. Graph Theory\/} 65 (2010), 216--231.
\\[-20pt] 
		
\bibitem{KovMMM09}  I. Kov\'acs, A. Malni\v{c}, D. Maru\v{s}i\v{c} and \v{S}. Miklavi\v{c},
One-matchng bi-Cayley graphs over abelian groups, 
{\em Europ. J. Combin.} 30 (2009), 602--616.
\\[-20pt] 


\bibitem{KutP1x}  K. Kutnar and P. Petecki,
On automorphisms and structural properties
of double generalized Petersen graphs, submitted.
\\[-20pt] 


\bibitem{Lu03}  Z.P. Lu, 
On the automorphism groups of bi-Cayley graphs, 
{\em Acta Sci. Nat. Univ. Peking\/} 39 (1) (2003), 1--5. 
\\[-20pt] 

\bibitem{LuWX06}  Z.P. Lu, C.Q. Wang and M.Y. Xu, 
Semisymmetric cubic graphs constructed from bi-Cayley graphs of $A_n$, 
{\em Ars Combin.} 80 (2006), 177--187.
\\[-20pt] 



\bibitem{Pis07}  T. Pisanski, 
A classification of cubic bicirculants, 
{\em Discrete Math.} 307 (2007), 567--578.
\\[-20pt] 


\bibitem{Sco64}  W.R. Scott,  
{\em Group theory}, 
Prentice-Hall, New Jersey 1964.
\\[-20pt] 

\bibitem{Wei84}  A.I. Weiss, An infinite graph of girth $12$, 
{\em Trans. Amer. Math. Soc.} 283 (1984), 575--588.
\\[-20pt] 

\bibitem{W08}  S. Wilson, 
Rose window graphs, 
{\em Ars Math. Contemp.} 1 (2008), 7--19 .
\\[-20pt] 


\bibitem{ZhouF12}  J.X. Zhou and  Y.Q. Feng, 
Cubic vertex-transitive non-Cayley graphs of order $8p$, 
{\em Electronic. J. Combin.} 19 (2012), \#53.
\\[-20pt] 

\bibitem{ZhouF14}  J.X. Zhou and Y.Q. Feng, 
Cubic bi-Cayley graphs over abelian groups, 
{\em Europ. J. Combin.} 36 (2014), 679--693.
\\[-20pt] 

}

\end{thebibliography}
\end{document}